\documentclass[12pt]{amsart}
\usepackage{latexsym, url}
\usepackage{amsthm}
\usepackage{amsmath}
\usepackage{amsfonts}
\usepackage{hyperref}
\usepackage{amssymb}
\usepackage[dvips]{graphicx}
\usepackage{xypic}
\input xy
\xyoption{all}

\newtheorem{theo}{Theorem}[section]
\newtheorem{lemma}[theo]{Lemma}

\newtheorem{propo}[theo]{Proposition}
\newtheorem{defi}[theo]{Definition}
\newtheorem{coro}[theo]{Corollary}
\newtheorem{rem}[theo]{Remark}
\newtheorem{con}[theo]{Construction}
\newtheorem{nota}[theo]{Notation}

\newtheorem{exam}[theo]{Example}
\newtheorem{exams}[theo]{Examples}

\newcommand\Ind{\operatorname{Ind}}

\newcommand\Ort{\operatorname{Ort}}

\newcommand\fd{\operatorname{fd}}

\newcommand\Mod{\operatorname{Mod}}
\newcommand\Gen{\operatorname{Gen}}

\newcommand\Alg{\operatorname{Alg}}

\newcommand\StrongMono{\operatorname{StrongMono}}
\newcommand\Iso{\operatorname{Iso}}
\newcommand\StrongEpi{\operatorname{StrongEpi}}
\newcommand\Epi{\operatorname{Epi}}
\newcommand\op{\operatorname{op}}
\newcommand\id{\operatorname{id}}

\newcommand\eps{\operatorname{\varepsilon}}
\newcommand\Set{\operatorname{\bf Set}}

\newcommand\Ban{\operatorname{\bf Ban}}
\newcommand\CAlg{\operatorname{\bf CAlg}}
\newcommand\CCAlg{\operatorname{\bf CCAlg}}

\newcommand\LG{\operatorname{\bf LG}}
\newcommand\N{\operatorname{\bf N}}
\newcommand\CMet{\operatorname{\bf CMet}}
\newcommand\Met{\operatorname{\bf Met}}
\newcommand\CCMet{\operatorname{\bf CCMet}}

\newcommand\pa{\operatorname{\parallel}}
\newcommand\Mono{\operatorname{Mono}}

\newcommand\colim{\operatorname{colim}}

\newcommand\ca{\mathcal {A}}
\newcommand\cb{\mathcal {B}}

\newcommand\cd{\mathcal {D}}

\newcommand\cg{\mathcal {G}}

\newcommand\ce{\mathcal {E}}

\newcommand\ck{\mathcal {K}}
\newcommand\cl{\mathcal {L}}
\newcommand\cm{\mathcal {M}}

\newcommand\cp{\mathcal {P}}

\newcommand\cv{\mathcal {V}}

 \newbox\noforkbox \newdimen\forklinewidth
\forklinewidth=0.3pt \setbox0\hbox{$\textstyle\smile$}
\setbox1\hbox to \wd0{\hfil\vrule width \forklinewidth depth-2pt
 height 10pt \hfil}
\wd1=0 cm \setbox\noforkbox\hbox{\lower 2pt\box1\lower
2pt\box0\relax}

\date{February 27, 2022}
 
\begin{document}
\title[Enriched locally generated categories]
{Enriched locally generated categories}
\author[I. Di Liberti and J. Rosick\'{y}]
{I. Di Liberti and J. Rosick\'{y}}
\thanks{The first author was supported by the Grant Agency of the Czech Republic project EXPRO 20-31529X and RVO: 67985840. The second author was supported by the Grant Agency of the Czech Republic under the grant 19-00902S} 
\address{
\newline I. Di Liberti\newline
Institute of Mathematics\newline
Czech Academy of Sciences\newline
\v{Z}itn\'{a} 25, Prague, Czech Republic\newline
diliberti.math@gmail.com\newline
\newline J. Rosick\'{y}\newline
Department of Mathematics and Statistics\newline
Masaryk University, Faculty of Sciences\newline
Kotl\'{a}\v{r}sk\'{a} 2, 611 37 Brno, Czech Republic\newline
rosicky@math.muni.cz
}
 
\begin{abstract}
We introduce the notion of $\cm$-locally generated category for a factorization system $(\ce,\cm)$ and study its properties. We offer a Gabriel-Ulmer duality for these categories, introducing the notion of nest. We develop this theory also from an enriched point of view. We apply this technology to Banach spaces showing that it is equivalent to the category of models of the nest of finite-dimensional Banach spaces.
\end{abstract} 

\maketitle
\tableofcontents
\section{Introduction}
Locally presentable categories were introduced by Gabriel and Ulmer \cite{GU} in 1971 and, since then, their importance was steadily growing. Today, they form an established framework to do category theory in the daily practice of the working mathematician. The reason of their success is merely evidence-based. On the one hand, they are technically handy, allowing transfinite constructions (e.g. the small object argument) and thus offering pleasing versions of relevant tools (e.g. the adjoint functor theorem). On the other hand, a vast majority of relevant categories happen to be locally presentable (with some very disappointing exceptions, like the category of topological spaces). 

Gabriel and Ulmer \cite{GU} also introduced locally generated categories. While locally presentable categories are based on the classical algebraic concept of a (finitely) presentable object, locally generated categories are based on (finitely) generated objects. This makes them easier to handle then locally presentable ones. Although the both classes of categories coincide, a category can be locally finitely generated without being locally finitely presentable. Despite this, locally generated categories were somewhat neglected. They were mentioned in \cite{AR} and, later, generalized in \cite{AR1}. In a nutshell, a $\cm$-locally generated category is a cocomplete category $\ck$, equipped with a factorization system $(\ce, \cm)$, that is generated by a set of $\cm$-generated objects. In \cite{GU} they were only dealing with (Strong Epi, Mono), in \cite{AR1}, monomorphisms were replaced by any right part of a proper factorization system $(\ce,\cm)$. 

We slightly generalize \cite{AR1} by taking any factorization system $(\ce,\cm)$. A locally $\lambda$-presentable category is then a $\ck^\to$-locally $\lambda$-generated category. More importantly, we extend the (nearly forgotten) Gabriel-Ulmer duality for locally generated categories to our setting. While the famous Gabriel-Ulmer duality for locally finitely presentable categories is based on the fact that they are sketched by a finite limit sketch, their duality for locally finitely generated categories uses finite limits and multiple pullbacks of monomorphisms. This explains why locally finitely generated categories do not need to be locally finitely presentable. In our situation, we use multiple pullbacks of morphisms from $\cm$. 

Our main goal is to apply these ideas to Banach spaces. For this, we have to make our theory enriched. While enriched locally presentable categories were treated by Kelly \cite{K1}, enriched locally generated categories have been never considered. Our motivation was \cite{AR2} where finite-dimensional Banach spaces, which are not finitely presentable, were shown to be finitely generated w.r.t. isometries when Banach spaces are enriched over $\CMet$, the category of complete metric spaces. We provide an appropriate framework for the observation \cite{AR2} 7.8 that Banach spaces are locally finitely generated w.r.t. isometries. The relevant factorization system is (dense maps, isometries). Our main accomplishment is that Banach spaces can be sketched using finite-dimensional Banach spaces equipped with finite weighted limits and multiple pullbacks of isometries. This has turned out to be more delicate than expected, because $\CMet$ is a quite ill-behaved enrichment base.

In Section 2 we introduce the notion of a locally generated category, providing examples and properties. For this, we need the concept of a $\lambda$-convenient factorization system $(\ce,\cm)$, which seems to be of its own interest. We discuss the main properties of $\lambda$-generated objects w.r.t. $\cm$. We provide a recognition theorem (\ref{char1}) for locally $\lambda$-generated categories which meets the spirit of the original definition of Gabriel and Ulmer. In Section 3 we introduce the notion of a $\lambda$-nest, a sketch-like gadget that we use to recast a suitable version of Gabriel-Ulmer duality for $\cm$-locally generated categories. We introduce the notion of a model of a nest, and show that the category of models of a nest is a $\cm$-locally generated category. In Sections 4 and 5 we present an enriched version of Sections 2 and 3, when the enrichment base $\cv$ is locally $\lambda$-presentable as a closed category.   Section 6 discusses the case of Banach spaces in detail, focusing on their axiomatizability via the $(\aleph_0, \CMet)$-nest $\Ban_{\text{fd}}$.

\section{Locally generated categories}
\textit{Locally $\lambda$-generated categories} were introduced by Gabriel and Ulmer \cite{GU} as cocomplete categories $\ck$ having a strong generator consisting of $\lambda$-generated objects such that every $\lambda$-generated object has only a set 
of strong quotients. Here, an object $A$ is \textit{$\lambda$-generated} if its hom-functor $\ck(A,-):\ck\to\Set$ preserves $\lambda$-directed colimits of monomorphisms.
A category is \textit{locally generated} if it is locally $\lambda$-generated
for some regular cardinal $\lambda$. \cite{GU} showed that locally generated categories coincide with locally presentable ones.
Every locally $\lambda$-presentable category is locally $\lambda$-generated but a locally $\lambda$-generated category does
not need to be locally $\lambda$-presentable. Thus the passage to locally generated categories can lower the defining cardinal
$\lambda$.

The definition of a locally $\lambda$-generated category was reformulated in \cite{AR}: a cocomplete category $\ck$ is locally
$\lambda$-generated if it has a set $\ca$ of $\lambda$-generated objects such that every object is a $\lambda$-directed colimit
of its subobjects from $\ca$. The fact that these two definitions are equivalent follows from \cite{AR} 1.70 and \cite{GU} 9.2.

A further step was made in \cite{AR1} where a cocomplete category $\ck$ with a proper factorization system $(\ce,\cm)$ was
called \textit{$\cm$-locally $\lambda$-generated} if it has a set $\ca$ of $\lambda$-generated objects w.r.t. $\cm$ such that every object is a $\lambda$-directed colimit of $\ca$-objects and $\cm$-morphisms. Here, an object $A$ is 
\textit{$\lambda$-generated w.r.t. $\cm$} if its hom-functor $\ck(A,-):\ck\to\Set$ preserves $\lambda$-directed colimits of 
$\cm$-morphisms. Again, \cite{AR1} proves that a category is $\cm$-locally generated for some proper
factorization system $(\ce,\cm)$ iff it is locally presentable. Moreover, locally $\lambda$-generated categories are
$\cm$-locally $\lambda$-generated ones for the factorization system (strong epimorphisms, monomorphisms). We extend this
definition a little bit.

\begin{defi}
{
\em 
Let $\ck$ be a category with a factorization system $(\ce,\cm)$ and $\lambda$ a regular cardinal. We say that an object $A$ is \textit{$\lambda$-generated w.r.t. $\cm$} if its hom-functor $\ck(A,-):\ck\to\Set$ preserves $\lambda$-directed colimits of $\cm$-morphisms.
}
\end{defi}

\begin{exams}
{
\em
(1) For the factorization system $(\Iso,\ck^\to)$, an object is $\lambda$-generated w.r.t. $\ck^\to$ iff it is $\lambda$-presentable.

(2) For the factorization system $(\ck^\to,\Iso)$, every object is $\lambda$-generated w.r.t. $\Iso$.  

(3) For the factorization system $(\StrongEpi,\Mono)$, an object is $\lambda$-generated w.r.t. $\Mono$ iff it is $\lambda$-generated.
}
\end{exams}

\begin{defi}\label{conv}
{
\em
Let $\lambda$ be a regular cardinal. A factorization system $(\ce,\cm)$ in a category $\ck$ will be called \textit{$\lambda$-convenient} if 
\begin{enumerate}
\item $\ck$ is $\ce$-cowellpowered, i.e., if every object of $\ck$ has only a set of $\ce$-quotients, and
\item $\cm$ is closed under $\lambda$-directed colimits, i.e., every $\lambda$-directed colimit of $\cm$-morphisms 
has the property that
\begin{enumerate}
\item a colimit cocone consists of $\cm$-morphisms, and
\item for every cocone of $\cm$-morphisms, the factorizing morphism is in $\cm$.
\end{enumerate}
\end{enumerate}
$(\ce,\cm)$ is \textit{convenient} if it is $\lambda$-convenient for some $\lambda$.
}
\end{defi}

\begin{rem}
{
\em
If $(\ce,\cm)$ is $\lambda$-convenient and $\lambda'\geq\lambda$ is regular then $(\ce,\cm)$ is $\lambda'$-convenient.
}
\end{rem}

\begin{defi}
\label{locgen}
{
\em
Let $\ck$ be a cocomplete category with a $\lambda$-convenient factorization system $(\ce,\cm)$ where $\lambda$ a is regular cardinal. We say that $\ck$ is \textit{$\cm$-locally $\lambda$-generated} if it has a set $\ca$ of $\lambda$-generated objects w.r.t. $\cm$ such that every object is a $\lambda$-directed colimit of objects from $\ca$ and morphisms from $\cm$. 

$\ck$ is called \textit {$\cm$-locally generated} if it is $\cm$-locally $\lambda$-generated for some regular cardinal $\lambda$.
}
\end{defi}

\begin{rem}\label{dense}
{
\em
In the same way as in II. (i) of the proof of \cite{AR} 1.70, we show that $\ca$ is dense in $\ck$.
Hence the canonical functor $E:\ck\to\Set^{\ca^{\op}}$ ( sending $K$ to $\ck(-,K)$ restricted on $\ca^{\op}$) is fully faithful. Since $\ck$ is cocomplete, $E$ is a right adjoint. Moreover, $E$ preserves $\lambda$-directed colimits of $\cm$-morphisms. Consequently, in a $\cm$-locally $\lambda$-generated category, $\lambda$-small limits commute with $\lambda$-directed colimits of $\cm$-morphisms.

}
\end{rem}

\begin{nota}
{
\em
In what follows, $\ck_\lambda$ will denote the full subcategory of $\ck$ consisting of $\lambda$-generated objects w.r.t. $\cm$ and by $\ce_\lambda$ the set $\ce\cap\ck_\lambda^\to$.
}
\end{nota}

\begin{exams}
{
\em
(1) The factorization system $(\Iso,\ck^\to)$ is convenient and $\ck$ is locally $\lambda$-presentable iff it is $\ck^\to$-locally
$\lambda$-generated.

(2) The category $\ck$ is $\Iso$-locally generated iff it is small. In fact, every object $K$ of $\ck$ has a morphism
$\coprod \ca\to K$ where the domain is the coproduct of all objects $A$ from $\ca$. Thus $(\ck^\to,\Iso)$ is convenient
iff $\ck$ is small.

(3) A $\Mono$-locally $\lambda$-generated category $\ck$ in the sense of \cite{GU} and \cite{AR} satisfies \ref{locgen} for the factorization system (strong epi, mono). Then this factorization system is $\lambda$-convenient. Indeed, following \ref{dense} the canonical functor $E:\ck\to\Set^{\ca^{\op}}$ is  fully faithful and preserves $\lambda$-directed colimits of monomorphisms. This implies that monomorphisms are closed in $\ck$
under $\lambda$-directed colimits (in the sense of \ref{conv}) because this holds in $\Set^{\ca^{\op}}$ and $E$ preserves and reflects monomorphisms. Now, we can show that $\ca$ is closed under strong quotients. Let $e:A\to B$ be a strong epimorphism with
$A\in\ca$, $m_i:K_i\to K$ a $\lambda$-directed colimit of monomorphisms and $f:B\to K$. Since $A$ is $\lambda$-generated,
there is $g:A\to K_i$ such that $m_ig=fe$. Since $m_i$ is a monomorphism, the diagonalization property yields $h:B\to K_i$
such that $m_ih=f$. Hence $B$ is $\lambda$-generated. Then II. in the proof of \cite{AR} 1.70 yields that $\ck$ is locally
presentable, hence cowellpowered. This verifies \ref{conv} (1), hence the (strong epi, mono) factorization system is 
$\lambda$-convenient. We gave this argument in detail because \cite{AR} is not accurate at this point.

(4) For a factorization system $(\Epi,\StrongMono)$, we get $\StrongMono$-locally $\lambda$-generated categories. Here, we have to assume that this factorization system is $\lambda$-convenient because it does not seem that we could get it for free like in (3).

(5) Any full reflective subcategory $\cl$ of $\ck$ determines a factorization system $(\ce,\cm)$ where $\ce$ consists of morphisms sent by the reflector to an isomorphism (see \cite{RT}). Both (1) and (2) are special cases for $\cl=\ck$ or $\cl$ consisting just from a terminal objects.
}
\end{exams}

\begin{exam}
{
\em
The category $\mathsf{CGWH}$ of compactly generated weakly Hausdorff spaces and continuous maps, has a factorization system (dense maps, closed embeddings). It is $\aleph_0$-convenient and \cite[3.8]{ST} shows that compact spaces are $\aleph_0$-generated with respect to this factorization system. \cite[1.11]{ST} can be taken as a reformulation of the fact that compact spaces form a dense subcategory of $\mathsf{CGWH}$. As a result, $\mathsf{CGWH}$ would be a class-locally generated category w.r.t. closed embeddings in the spirit of \cite{CR}, if only we had defined such notion. Yet, its behavior is incredibly similar to that of a locally generated category, as the only obstruction is that $\aleph_0$-generated objects form a proper class. This example finally clarifies the sense in which compact spaces are \textit{compact} objects, as the classical terminology of algebraic topology addresses finitely presentable objects.
}
\end{exam}

\begin{exam}
{
\em
A similar example is given by $\Delta$-generated spaces \cite{FR}, for the same choice of a factorization system. In this case though, the category is indeed locally finitely generated w.r.t. closed embeddings.
}
\end{exam}


\begin{nota}
{
\em
We say that a category is \textit{locally $\lambda$-generated} if it is $\cm$-locally $\lambda$-generated for some $\cm$. Locally $\lambda$-generated categories in the sense of \cite{GU} are then $\Mono$-locally $\lambda$-generated.
}
\end{nota}

\begin{lemma}\label{quot}
$\lambda$-generated objects w.r.t. $\cm$ are closed under $\ce$-quotients.
\end{lemma}
\begin{proof}
Let $e:A\to B$ is in $\ce$ and $A\in\ca$. Consider a $\lambda$-directed colimit $m_i:K_i\to K$, $i\in I$ of $\cm$-morphisms and $f:B\to K$. Since $A$ is $\lambda$-generated w.r.t. $\cm$, there is $g:A\to K_i$ such that $m_ig=fe$. Since $m_i$ is in $\cm$, 
the diagonalization property yields $h:B\to K_i$ such that $m_ih=f$. Assume that $m_ih'=f$ for $h':B\to K_i$.
Then $m_ihe=m_ih'e$ and, since $A$ is $\lambda$-generated w.r.t. $\cm$, there is $i\leq j\in I$ such that $m_{ij}he=m_{ij}h'e$.
Hence both $m_{ij}h$ and $m_{ij}h'$ are diagonals in the commutative square
$$
		\xymatrix@=2pc{
			A \ar [r]^{e}\ar[d]_{m_{ij}he} & B \ar[d]^{f}\\
			K_j  \ar [r]_{m_j}& K
		}
		$$ 
		and thus they are equal. Hence $B$ is $\lambda$-generated.

\end{proof}

\begin{lemma}\label{gen}
In an $\cm$-locally $\lambda$-generated category $\ck$, $\lambda$-generated objects w.r.t. $\cm$ are closed under $\lambda$-small colimits and, up to isomorphism, form a set.
\end{lemma}
\begin{proof}
Let $\colim A_j$ be a $\lambda$-small colimit of $\lambda$-generated objects w.r.t. $\cm$ and $\colim K_i$ a $\lambda$-directed colimit of $\cm$-morphisms. Since $\lambda$-directed colimits commute in $\Set$ with $\lambda$-small limits, we have
 \begin{align*}
  \ck(\colim A_j,\colim K_i) &\cong \lim_j \ck(A_j,\colim K_i) \\
    &\cong \lim_j\colim_i\ck(A_j, K_i) \\
  &\cong \colim_i\lim_j \ck(A_j,K_i)  \\
 &\cong \colim_i \ck(\colim_j A, K_i)\\
   \end{align*}
Hence $\colim A_j$ is $\lambda$-generated w.r.t. $\cm$.
   
The second claim follows from the fact that every $\lambda$-generated object w.r.t. $\cm$ is is a $\lambda$-directed colimit of $\ca$-objects and $\cm$-morphisms, which makes it a retract of an $\ca$-object, i.e., a finite colimit of $\ca$-objects. Indeed, given a retract $u:B\to A$ with a section $s:A\to B$, then $s$ is a coequalizer of $us$ and $\id_A$.
\end{proof}

\begin{rem}\label{gen1}
{
\em
For the first claim, we do not need to assume that $\ck$ is $\cm$-locally $\lambda$-generated.
}
\end{rem}

\begin{nota}
{
\em
Given a family of arrows $\ce$, we call $\ce^\perp$ the family of arrows that are orthogonal to $\ce$. This means that an arrow $m$ belongs to $\ce^\perp$ if it has the unique right lifting property with respect to every arrow in $\ce$. An object $A$ is \textit{orthogonal} to $\ce$ if the unique morphism from $A$ to the terminal object is in $\ce^\perp$.
}
\end{nota}
\begin{lemma}\label{ort}
If $\ck$ is $\cm$-locally $\lambda$-generated then $\cm=(\ce_\lambda)^\perp$.
\end{lemma}
\begin{proof}
Every $e:A\to B$ in $\ce$ is a $\lambda$-directed colimit $(a_i,b_i):f_i\to e$ in $\ck_\lambda^\to$. Every $f_i:A_i\to B_i$ has the factorization $f_i=m_ie_i$ with 
$e_i:A_i\to C_i$ in $\ce$ and $m_i:C_i\to B_i$ in $\cm$. Then $e=\colim m_i\cdot\colim e_i$
where $\colim e_i:A\to C$ and $\colim m_i:C\to B$. Since $C=\colim C_i$ is a $\lambda$-directed
colimit of $\cm$-morphisms and $(\ce,\cm)$ is $\lambda$-convenient, $\colim m_i$ is in $\cm$
and thus it is an isomorphism. Hence $e=\colim e_i$. Consequently $\cm=\ce_\lambda^\perp$.
\end{proof}

\begin{rem}\label{conv1}
{
\em
We have shown that, in a $\cm$-locally $\lambda$-generated category $\ck$, $\ce$ is the closure $\colim(\ce_\lambda)$ of $\ce_\lambda$ under all colimits in $\ck^\to$.

Conversely, if $\cm=\ce_\lambda^\perp$ then $\cm$ satisfies \ref{conv}(2).
}
\end{rem}

The proof of the next theorem follows \cite{AR1} and, like \cite{AR1}, is based on \cite{GU}. However, \cite{AR1} lacks
the assumption that $(\ce,\cm)$ is convenient.

\begin{theo}\label{char}
For every category $\ck$ equivalent are:
\begin{enumerate}
\item $\ck$ is locally presentable,
\item $\ck$ is $\cm$-locally generated for some convenient factorization system $(\ce,\cm)$, and
\item $\ck$ is $\cm$-locally generated for every convenient factorization system $(\ce,\cm)$.
\end{enumerate}
\end{theo}
\begin{proof}
Clearly, $(3)\Rightarrow (1)\Rightarrow (2)$. The implication $(1)\Rightarrow (3)$ is analogous to I. of the proof \cite{AR} 1.70.

$(2)\Rightarrow (1)$ follows \cite{AR1}. In fact, let $\ck$ be $\cm$-locally $\lambda$-generated. Then $E:\ck\to\Set^{\ca^{\op}}$ makes $\ck$ equivalent to a full reflective subcategory of $\cl=\Set^{\ca^{\op}}$ closed under $\lambda$-directed colimits of $\cm$-morphisms (see \ref{dense}). Let $\cp$ consist of reflections of $\cl$-objects in $\ck$ which are either $\lambda$-small colimits in $\cl$ of diagrams in $\ca$ or codomains of multiple pushouts in $\cl$ of $\ce$-morphisms with a domain in $\ca$. Observe that $\cp$ is small: the case of $\lambda$-small colimits is clear, for the multiple pushout use \ref{quot} and the fact that every $\lambda$-generated object w.r.t. $\cm$ is a retract of an $\ca$-object. Thus, the class $\Ort(\cp)$ of all $\cl$-objects orthogonal to $\cp$ is a locally presentable category. Moreover, $\ck$ is closed in $\Ort(\cp)$ under
\begin{enumerate}
\item $\lambda$-small colimits of diagrams in $\ca$,
\item multiple pushouts of $\ce$-morphisms with a domain in $\ca$, and
\item $\lambda$-directed colimits of $\cm$-morphisms.
\end{enumerate}

Consider a morphism $f:A\to L$ with $A$ in $\ca$ and $L$ in $\Ort(\cp)$. Let $e_f:A\to K_f$ be the cointersection in $\ck$ of all $\ce$-morphisms through which $f$ factors. Following (2) above,
there is $m_f:K_f\to L$ such that $f=m_fe_f$. Clearly, $e_f$ is in $\ce$ and $m_f$ is $\ce$-extremal in the sense that any $\ce$-morphism $K_f\to K$ through which $m_f$ factors is an isomorphism. Moreover, this factorization is "functorial": given $f':A'\to L$ and $h:A\to A'$ with $f=f'h$ there exists $h^\ast:K_f\to K_{f'}$ in $\cm$ with $m_f=m_{f'}h^\ast$. In fact, form 
a pushout 
$$
		\xymatrix@=2pc{
			A \ar [r]^{e_{f'}h}\ar[d]_{e_f} & K_{f'} \ar[d]^{\tilde{e}_f}\\
			K_f  \ar [r]_{h^\ast}& K
		}
		$$ 
There is $g:K\to L$ such that $gh^\ast=m_f$ and $g\tilde{e}_f=m_{f'}$. Since $\tilde{e}_f$ is in $\ce$ and $m_{f'}$ is $\ce$-extremal, $\tilde{e}_f$ is an isomorphism. 	

Every $L$ in $\Ort(\cp)$ is a $\lambda$-directed colimit $d_i:D_i\to L$ where $D_i$ are $\lambda$-small colimits of diagrams in $\ca$. Form reflections $r_i:D_i\to D'_i$ in $\ck$. Since $r_i\in\cp$, $d_i$ factors through $r_i$, say, $d_i=d_i'r_i$. Since $D_i'$ is in $\ca$, we can take the factorization $d_i'=m_ie_i$ above with $m_i:D_i''\to L$. Since $m_i$ is $\ce$-extremal, $d_i''$ is in $\cm$. Thus $m_i:D_i''\to L$ is a $\lambda$-directed colimit of $\cm$-morphisms, which implies that $L$ is in $\ck$. 

We have proved that $\ck=\Ort(\cp)$, hence $\ck$ is locally presentable. 	
\end{proof}

\begin{rem}\label{lambda}
{
\em
(1) All implications, except $(2)\Rightarrow (1)$, are valid for a given $\lambda$. In $(2)\Rightarrow (1)$, $\lambda$ can increase.

(2) The previous theorem sheds a light on the notion of locally generated category. Indeed, the framework of locally generated categories is not more expressive than that of locally presentable ones. Yet, as pointed out (1), it might happen that a category is locally finitely generated without being locally finitely presentable. From the point of view of essentially algebraic theories (see 3.D in \cite{AR}), this means that some $\lambda$-ary essentially algebraic theories can be axiomatized by the data of a finite limit theory and a factorization system. Thus, the complexity of infinitary (partial) operations can be hidden under the carpet of the factorization system. From a less technical perspective, the factorization system that makes the category locally generated is very often an extremely natural one to consider. A sketch-oriented  (see 2.F in \cite{AR}) interpretation of this discussion is the core motivation for the forthcoming notion of $\lambda$-nest.
}
\end{rem}

\begin{rem}
{
\em
We do not know whether, given a locally $\lambda$-presentable category $\ck$, one can find an $\aleph_0$-convenient factorization system $(\ce,\cm)$ such that $\ck$ is $\cm$-locally finitely generated. Following \ref{dense}, finite limits should commute in $\ck$ with directed colimits of $\cm$-morphisms.
}
\end{rem}
 
\begin{nota}
{
\em
$\Gen_\lambda(\ck)$ will denote the full subcategory of $\lambda$-generated objects w.r.t. $\cm$ in a $\cm$-locally $\lambda$-generated category $\ck$.
}
\end{nota}

\begin{rem}
{
\em
A multiple pullback $P$ is a limit of a diagram consisting of morphisms $f_i:A_i\to A$, $i\in I$. We can well-order $I$ as $\{i_0,i_1,\dots,i_j\dots\}$ and form pullbacks $P_j$ as follows: $P_0$ is the pullback 
$$
		\xymatrix@=2pc{
			P_0 \ar [r]^{\bar{f}_0}\ar[d]_{\bar{f}_1} & A_0 \ar[d]^{f_0}\\
			A_1  \ar [r]_{f_1}& A
		}
$$ 	
 
Then $P_1$ is the pullback	
$$
		\xymatrix@=2pc{
			P_1 \ar [r]^{p_{01}}\ar[d]_{} & P_0 \ar[d]^{f_0\bar{f}_1}\\
			A_2  \ar [r]_{f_2}& A
		}
        $$ 
We proceed by recursion and in limit steps we take limits. In this way, we transform multiple pullbacks to limits of smooth well-ordered chains (\textit{smooth} means that in limit steps we have limits).  

Conversely, a limit of a smooth chain $p_{ij}:P_j\to P_i$, $i\leq j<\alpha$ is a multiple pullback of $p_{0j}$, $j\in I$. Moreover, if $\ck$ is equipped with a factorization system $(\ce,\cm)$ then multiple pullbacks of $\cm$-morphisms correspond to limits of smooth well-ordered chains of $\cm$-morphisms. Indeed, $\cm$ is stable under pullbacks.
}
\end{rem}

\begin{coro}\label{sketch}
Let $\ck$ be an $\cm$-locally $\lambda$-generated category. Then $\ck $ is equivalent to the full subcategory of $\Set^{\Gen_\lambda(\ck)^{\op}}$ consisting of functors preserving $\lambda$-small limits and sending multiple pushouts of $\ce$-morphisms to multiple pullbacks.
\end{coro}
\begin{proof}
Following \cite{AR} 1.33(8), this describes $\Ort(\cp)$ from the proof of \ref{char}.
\end{proof}

The following theorem makes easier to identify $\cm$-locally generated categories among cocomplete categories with a factorization system and goes back to \cite{GU}.

\begin{theo}\label{char1}
A cocomplete category $\ck$ equipped with a $\lambda$-convenient factorization system $(\ce,\cm)$ is $\cm$-locally $\lambda$-generated iff it has a strong generator formed by $\lambda$-generated objects w.r.t. $\cm$.
\end{theo}
\begin{proof}
Using \ref{gen} and \ref{quot}, the closure $\ca$ of our strong generator under $\lambda$-small colimits and $\ce$-quotients consists of $\lambda$-generated objects w.r.t. $\cm$. Following the proof of \cite{AR} 1.11, we show that every object in $\ck$ is a $\lambda$-filtered colimit of $\ca$-objects. Using $(\ce,\cm)$-factorizations, we get that every object of $\ck$ is a $\lambda$-directed colimit of $\ca$-objects and $\cm$-morphisms.
\end{proof}

\begin{rem}
{
\em
If $\lambda'\geq\lambda$ and $\ck$ is $\cm$-locally $\lambda$-generated then $\ck$ is $\cm$-locally $\lambda'$-generated.
}
\end{rem}

\begin{theo}\label{monad}
Let $\ck$ be a $\cm$-locally $\lambda$-generated category where $(\ce,\cm)$ is a proper $\lambda$-convenient factorization system and let $T$ be a monad preserving $\cm$-morphisms and $\lambda$-directed colimits of $\cm$-morphisms. Then, assuming Vop\v enka's principle, the category of algebras $\mathsf{Alg}(T)$ is locally $\lambda$-generated.
\end{theo}
\begin{proof}
Following \cite{Ba} 3.3, $\Alg(T)$ is cocomplete. Consider the adjunction 
$F: \ck \leftrightarrows \mathsf{Alg}(T): U$ and put $\cm'=U^{-1}(\cm)$. Using \cite{B} 4.3.2, we get that $\cm'$ is closed under $\lambda$-directed colimits. Following \cite{Ri} 11.1.5, $\cm'=F(\ce)^\perp$. Since $\ce=\colim\ce_\lambda$ (see \ref{conv1}), we have 
$\colim F(\ce)=\colim(F(\ce_\lambda))$. $U$ is clearly conservative (and faithful), thus $F$ maps the dense generator $\Gen_\lambda(\ck)$ to a strong generator $\cg$ in $\mathsf{Alg}(T)$. Let $\ca$ be the closure of $\cg$ under $\lambda$-small colimits. Following \ref{gen1}, $\ca$ is dense in $\ck$. Assuming Vop\v enka's principle, $\Alg(T)$ is locally presentable (see \cite{AR} 6.14). Hence, following \cite{FR} 2.2, $(\colim F(\ce),\cm')$ is a factorization system on $\Alg(T)$. Since $\Alg(T)$ is co\-wellpowered (see \cite{AR} 1.58) and $\colim F(\ce)\subseteq\Epi$, this factorization system $(\colim F(\ce),\cm')$ is $\lambda$-convenient. We conclude the proof by the previous theorem and \ref{preservesmall}.
\end{proof}

\begin{rem}
{
\em
We do not know whether Vop\v enka's principle is really needed. This is related to the Open Problem 3 in \cite{AR}. In fact, let $\cl$ be a full reflective subcategory of a locally $\lambda$-presentable category $\ck$ closed under $\lambda$-directed colimits of monomorphisms. Then the monad $T=FG$, where $G:\cl\to\ck$ is the inclusion and $F$ its left adjoint, preserves $\lambda$-directed colimits of monomorphisms. Since $\cl\cong\Alg(T)$, \ref{monad} without Vop\v enka's principle would yield a positive solution of the Open Problem 3.
}
\end{rem}

\section{Extended Gabriel-Ulmer duality}
The Gabriel-Ulmer duality is a contravariant biequivalence between categories with $\lambda$-small limits and locally $\lambda$-presentable categories (see \cite{GU} 7.11 or \cite{AR} 1.45). We are going to to extend this duality to $\cm$-locally $\lambda$-generated categories. It will also cover the Gabriel-Ulmer duality for $\Mono$-locally $\lambda$-generated categories. In order to do so, we will introduce the notion of a nest.

\begin{defi}
{
\em
A \textit{$\lambda$-nest} is a small category $\ca$ equipped with a factorization system $(\ce_\ca,\cm_\ca)$ and having $\lambda$-small limits and multiple pullbacks of $\cm_\ca$-morphisms.
}
\end{defi}

\begin{rem}
{
\em
$\lambda$-nests $\ca$ with $(\Iso,\ca^\to)$ are precisely small categories with $\lambda$-small limits. $\lambda$-nests with $(\StrongEpi,\Mono)$ are precisely small "echt" $\lambda$-complete categories of \cite{GU}.
}
\end{rem}

\begin{exam}
{
\em
$\Gen_\lambda(\ck)^{\op}$ is a $\lambda$-nest for every $\cm$-locally $\lambda$-generated category $\ck$. This follows from \ref{gen}, \ref{quot} and the fact that $(\cm_\ca\cap\ca^\to,\ce_\ca\cap\ca^\to)$ is a factorization system on $\Gen_\lambda(\ck)^{\op}$.
}
\end{exam}

\begin{nota}
{
\em
For a $\lambda$-nest $\ca$, $\Mod_\lambda(\ca)$ denotes the category of all \textit{models}, i.e., of functors $\ca\to\Set$ preserving $\lambda$-small limits and multiple pullbacks of $\cm$-morphisms.
}
\end{nota}

\begin{lemma}\label{fact}
$\Mod_\lambda(\ca)$ is a locally presentable category equipped with a factorization system $(\ce,\cm)$.
\end{lemma}
\begin{proof}
$\Mod_\lambda(\ca)$ is a full subcategory of $\Set^\ca$ and the codomain restriction of the Yoneda embedding $Y:\ca^{\op}\to\Set^\ca$ is a full embedding of $\ca^{\op}$ to $\Set^\ca$. 
Following \cite{AR} 1.51, $\Mod_\lambda(\ca)$ is locally presentable. The factorization system
on $\ca^{\op}$ is $(\cm_\ca,\ce_\ca)$ and the factorization system $(\ce,\cm)$ is given, following \cite{FR} 2.2, as: $\ce=\colim(\cm_\ca)$ and $\cm=\cm_\ca^\perp$.  
\end{proof}

\begin{theo}\label{model}
$\Mod_\lambda(\ca)$ is a $\cm$-locally $\lambda$-generated category for every $\lambda$-nest $\ca$.
\end{theo}
\begin{proof}
I. At first, we will show that $\ck=\Mod_\lambda(\ca)$ is closed in $\cl=\Set^\ca$ under $\lambda$-directed colimits of $\cm$-morphisms. Following \cite{AR} 1.33(8), $\ck=\Ort(\cp)$ where $\cp$ is defined as in the proof of \ref{char}. This means that $\cp$ consists of reflections of $\cl$-objects in $\ck$ which are either $\lambda$-small colimits of diagrams in $\ca^{\op}$ or codomains of multiple pushouts of $\cm_\ca$-morphisms in $\ca^{\op}$. 

Let $k_i:K_i\to K$ be a colimit in $\cl$ of a $\lambda$-directed diagram of $\cm$-morphisms in $\ck$. Let $r:\colim A_j\to A$ be a reflection of a $\lambda$-small colimit in $\cl$ of a diagram in $\ca^{\op}$; its reflection in $\ck$ lies in $\ca^{\op}$. It is easy to see that objects orthogonal to $r$ are closed under $\lambda$-directed colimits; in fact, these objects correspond to functors $\ca\to\Set$ preserving $\lambda$-small limits. 

Let $r:P\to A$ be a reflection of a multiple pushout of $\cm_\ca$-morphisms. Alternatively, $P$ can be seen as a colimit of a well ordered smooth chain   
$$
P_0 \xrightarrow{\ p_{01}\ } P_1 \xrightarrow{\ p_{12}\ } P_2 \xrightarrow{\ p_{23}\ }\dots
$$
of $\cm_\ca$-morphisms. Let $f:P\to K$. There exists $i_0$ and $g_0:P_0\to K_{i_0}$ such that $k_{i_0}g_0=fp_0$. There exists $i_1>i_0$ and $g'_1:P_1\to K_{i_1}$ such that $k_{i_0i_1}g_0=g'_1p_{01}$. Since $k_{i_0i_1}\in\cm_\ca^\perp$ and $p_{01}\in\cm_\ca$, there is $g_1:P_1\to K_{i_0}$ such that $g_1p_{01}=g_0$ and $k_{i_0i_1}g_1=g'_1$. Continuing this procedure by taking colimits in limit steps, we get a cocone $g_j:P_j\to K_{i_0}$ inducing $g:P\to K_{i_0}$ such that $k_{i_0}g=f$. There is $h:A\to K_{i_0}$ with $hr=g$. Hence $k_{i_0}hr=f$. The uniqueness of this extension follows from $A$ being in $\ca^{\op}$. 

II. From I., it follows that every object of $\ca^{\op}$ is $\lambda$-generated w.r.t. $\cm$ in $\ck$. Following \ref{fact} and \ref{conv1}, $\cm$ satisfies \ref{conv}(2). In the same way as in \ref{quot}, we show that every $\ce$-quotient of an $\ca^{\op}$-object is $\lambda$-generated w.r.t. $\cm$. Since every $\lambda$-generated object w.r.t. $\cm$ is a retract of an $\ca^{\op}$-object (cf. \ref{gen}) and $\ca^{\op}$ is closed under retracts, $\ca^{\op}=\Gen_\lambda(\ck)$. Using $(\ce,\cm)$ factorizations, we show that every object in $\ck$ is a $\lambda$-directed colimit of $\ca^{\op}$-objects.

It remains to show that $\ck$ is $\ce$-cowellpowered. Let $e:K\to L$ be an $\ce$-quotient of $K$. Then $e$ is a $\lambda$-directed colimit of $\ce$-morphisms $e_i:K_i\to L_i$ where $K_i$ and $L_i$ are in $\ca^{\op}$. Since there is only set of expressions of $K$ as a $\lambda$-directed colimit of $\ca^{\op}$-objects, there is only a set of $\ce$-quotients of $K$. 
\end{proof}

\begin{rem}\label{eq}
{
\em
We have 
$$
\ca\simeq (\Gen_\lambda(\Mod_\lambda(\ca)))^{\op}
$$ 
for every $\lambda$-nest $\ca$, and
$$
\ck\simeq\Mod_\lambda(\Gen_\lambda(\ck)^{\op})
$$
for every $\cm$-locally $\lambda$-generated category $\ck$.

Note that these equivalences also include the corresponding factorization systems. In the first case, $(\colim(\cm_\ca),\cm_\ca^\perp)$ restricts to $(\cm_\ca,\ce_\ca)$ (see \ref{fact}) and, in the second case, it follows from \ref{conv1}.
}
\end{rem}

\begin{rem}\label{ind}
{
\em
For a $\lambda$-nest $\ca$, $\Mod_\lambda(\ca)$ is a full subcategory of the free completion $\Ind_\lambda(\ca)$ of $\ca$ under $\lambda$-directed colimits. Indeed, the latter category consists of functors $\ca^{\op}\to\Set$ preserving $\lambda$-small limits.
}
\end{rem}

\begin{defi}
{
\em
A \textit{morphism of locally $\lambda$-generated categories} $R: \ck \to \cl$ is a right adjoint preserving $\cm$-morphisms and $\lambda$-directed colimits of them. 
}
\end{defi}

\begin{rem}\label{adjoint} 
{
\em
$R$ preserves $\cm$ if and only if its left adjoint $L:\cl\to\ck$ preserves $\ce$ (see \cite{Ri} 11.1.5).
}
\end{rem}

\begin{lemma}\label{preservesmall}
Let $R: \ck \to \cl$ be a morphism of locally $\lambda$-generated categories. Then its left adjoint $L$ preserves $\lambda$-generated objects w.r.t. $\cm$.
\end{lemma}
\begin{proof}
Let $A$ be $\lambda$-generated w.r.t. $\cm$ in $\cl$. Consider a directed colimit $\colim K_i$ of $\cm$-morphisms in $\ck$. Then

\begin{align*}
  \ck(LA,\colim K_i) &\cong \cl(A,R\colim K_i) \\
    &\cong \cl(A, \colim RK_i) \\
  &\cong \colim \cl(A,K_i)  \\
   &\cong \colim \ck(LA, K_i). 
\end{align*}

\end{proof}

\begin{defi}
{
\em
A \textit{morphism of $\lambda$-nests} $F: \ca \to \cb$ is a functor preserving $\lambda$-small limits, $\cm$-morphisms and multiple pullbacks of them.
}
\end{defi}

\begin{nota}
{
\em
Let $\LG_\lambda$ be the $2$-category of locally $\lambda$-generated categories and $\N_\lambda$ be the $2$-category of $\lambda$-nests, in the both cases $2$-cells are natural transformations.
}
\end{nota}

\begin{con}[The functor $\Gen_\lambda$]
{
\em
Given a locally $\lambda$-generated category $\ck$, we have defined $\Gen_\lambda(\ck)^{\op}$ to be the opposite category of its $\lambda$-generated objects. It is easy to see that this construction is (contravariantly) functorial. Indeed, given a morphism of locally $\lambda$-generated categories $R: \ck \to \cl$, its left adjoint $L$ restricts to $\lambda$-generated objects \[ L: \Gen_\lambda(\cl) \to \Gen_\lambda(\ck)\] by \ref{preservesmall}, and passing to the opposite category, is a morphism of $\lambda$-nests because of \ref{quot} and \ref{gen}.
}
\end{con}

\begin{con}[The functor $\Mod_\lambda$]\label{models}
{
\em
Given a $\lambda$-nest $\ca$, we have seen that the category $\Mod_\lambda(\ca)$ is locally $\lambda$-generated. We will extend this construction to a (contravariant) functor. Given a morphism of $\lambda$-nests $F: \ca \to \cb$, the functor $\Mod_\lambda(F)$ sends $H$ from $\Mod_\lambda(\cb)$ to $HF$. This functor is the domain restriction of the functor $\Ind_\lambda(\cb)\to\Ind_\lambda(\ca)$ given, again, by precompositions with $F$. The latter functor has the left adjoint $\Ind_\lambda(F):\Ind_\lambda(\ca)\to\Ind_\lambda(\cb)$. The domain restriction $L(F)$ of this left adjoint is a left adjoint to $\Mod_\lambda(F)$. The functor $L(F)$ preserves $\ce$-morphisms because
$$
L(F)(\ce_{\Mod_\lambda(\ca)})=L(F)(\colim\cm_\ca)\subseteq\colim(F(\cm_\ca))\subseteq
\ce_{\Mod\lambda}(\cb).
$$
Thus $\Mod_\lambda(F)$ preserves $\cm$-morphisms and, therefore, it is a morphism of $\lambda$-generated categories.
}
\end{con}

\begin{theo}\label{dual}
\[ \Gen_\lambda:\LG _\lambda \leftrightarrows \N^{\op}_\lambda : \Mod_\lambda \] is a dual biequivalence between locally $\lambda$-generated categories and $\lambda$-nests.
\end{theo}
\begin{proof}
It follows from \ref{eq} and \ref{models}.
\end{proof}

\begin{rem}\label{duality}
{
\em
(1) Our duality restricts to the standard Gabriel-Ulmer duality between locally $\lambda$-presentable categories and small categories with $\lambda$-small limits (for $\cm=\Iso$) and, a little bit forgoten, Gabriel-Ulmer duality for locally $\lambda$-generated categories (see \cite{GU} 9.8). In the second case $\cm=\Mono$ and the dual is formed by small categories with $\lambda$-small limits and pullbacks of strong monomorphisms.

On the other hand, our extended Gabriel-Ulmer duality is a restriction of the standard one (see \ref{models}).

(2) Like the Gabriel-Ulmer duality (see \cite{MP}), our duality is given by the category $\Set$ being both a large $\lambda$-nest and a locally $\lambda$-generated category. Clearly, models of a $\lambda$-nest $\ca$ are $\lambda$-nest morphisms $\ca\to\Set$. Conversely, a morphism $U:\ck\to\Set$ of locally $\lambda$-generated categories is uniquely determined by its left adjoint $F:\Set\to\ck$, these restrictions uniquely corespond to objects in $\Gen_\lambda(\ck)$. 

(3) \cite{CV} generalized the Gabriel-Ulmer duality to certain limit doctrines. This is based on the commutation of certain limits and colimits in $\Set$. But our duality (even that for $\Mono$-locally $\lambda$-generated categories) does not fall under this scope.
}
\end{rem}  

\section{Enriched locally generated categories}
In what follows, $\cv$ will be a complete and cocomplete symmetric monoidal closed category.
We will work with $\cv$-categories and under a $\lambda$-directed colimit we will mean a conical $\lambda$-directed one. $\cv$-factorization systems were introduced in \cite{D} and their theory was later developed by Lucyshyn-Wright in \cite{LW}. We refer to \cite{LW} for the main definitions and notations.

\begin{defi}
{
\em 
Let $\ck$ be a $\cv$-category with a $\cv$-factorization system $(\ce,\cm)$ and $\lambda$ a regular cardinal. We say that an object $A$ is \textit{$\lambda$-generated w.r.t. $\cm$} if its hom-functor $\ck(A,-):\ck\to\cv$ preserves $\lambda$-directed colimits of $\cm$-morphisms.
}
\end{defi}

\begin{rem}\label{V-fact}
{
\em
In a tensored $\cv$-category $\ck$, $\cv$-factorization systems are precisely factorization systems
$(\ce,\cm)$ such that $\ce$ is closed under tensors (see \cite{LW} 5.7).
}
\end{rem}

\begin{defi}\label{V-conv}
{
\em
A $\cv$-factorization system $(\ce,\cm)$ in a tensored $\cv$-category $\ck$ is called \textit{$\lambda$-convenient} if it is $\lambda$-convenient as an (ordinary) factorization system.
}
\end{defi}

\begin{defi}\label{V-locgen}
{
\em
Let $\ck$ be a cocomplete $\cv$-category with a $\lambda$-convenient $\cv$-fac\-to\-ri\-za\-tion system $(\ce,\cm)$ where $\lambda$ a is regular cardinal. We say that $\ck$ is \textit{$\cm$-locally $\lambda$-generated} if it has a set $\ca$ of $\lambda$-generated objects w.r.t. $\cm$ such that every object is a $\lambda$-directed colimit of objects from $\ca$ and morphisms from $\cm$. 

$\ck$ is called \textit {$\cm$-locally generated} if it is $\cm$-locally $\lambda$-generated for some regular cardinal $\lambda$.
}
\end{defi}

\begin{exam}
{
\em
For a cocomplete $\cv$-category $\ck$, $(\Iso,\ck^\to)$ is a convenient $\cv$-fac\-to\-ri\-za\-tion system. $\ck^\to$-locally $\lambda$-generated categories are locally presentable $\cv$-categories in the sense of \cite{K1}.
}
\end{exam}

\begin{rem}\label{V-dense}
{
\em
In \ref{V-locgen}, $\ca$ is dense in $\ck$, i.e., the canonical $\cv$-functor $E:\ck\to\cv^{\ca^{\op}}$ is fully faithful. In fact, $E$ preserves $\lambda$-directed colimits of $\cm$-morphisms and, then, the result follows from \cite{K}, 5.19.
}
\end{rem}

\begin{rem}\label{V-arrow}
{
\em
For a $\cv$-category $\ck$, the arrow category $\ck^\to$ is a $\cv$-category with $\ck^(f,g)$ defined by the following pullback in $\cv$
$$
		\xymatrix@=3pc{
			\ck^\to(f,g) \ar [r]^{}\ar[d]_{} & \ck(A,C) \ar[d]^{\ck(A,g) }\\
			 \ck(B,D)  \ar [r]_{\ck(f,D)}& \ck(A,D)
		}
		$$ 
where $f:A\to B$ and $g:C\to D$.	
}
\end{rem}
	
\begin{propo}\label{V-arrow1}
Assume that pullbacks commute with $\lambda$-directed colimits in $\cv$. Then, for a $\cm$-locally $\lambda$-generated $\cv$-category $\ck$, the $\cv$-category $\ck^\to$ is $\cm_\to$-locally $\lambda$-generated.
\end{propo}
\begin{proof}
Following \cite{Ri} 13.1, $\ck^\to$ is tensored and cotensored. Since $\ck^\to$ has limits and colimits (calculated pointwise), it is complete and cocomplete (see \cite{B} 6.6.16). The factorization system $(\ce_\to,\cm_\to)$ on $\ck^\to$ is defined pointwise from that on $\ck$ and is clearly $\lambda$-convenient. We will show that any morphism $f:A\to B$ with $A$ and $B$ $\lambda$-generated w.r.t. $\cm$ is $\lambda$-generated w.r.t. $\cm_\to$.

Consider a $\lambda$-directed colimit $g_i\to g$ on $\cm_\to$-morphisms in $\ck^\to$. We have pullbacks
$$
		\xymatrix@=3pc{
			\ck^\to(f,g_{g_i}) \ar [r]^{}\ar[d]_{} & \ck(A,C_i) \ar[d]^{\ck(A,g_i) }\\
			 \ck(B,D_i)  \ar [r]_{\ck(f,D_i)}& \ck(A,D_i)
		}
		$$ 
Since pullbacks commute with $\lambda$-directed colimits in $\cv$, 
$$
\ck^\to(f,g)\cong \colim\ck^\to(f,g_i).
$$		
Clearly, every $h$ in $\ck^\to$ is a $\lambda$-directed colimit of $\lambda$-generated objcts w.r.t. $\cm_\to$ and $\cm_\to$-morphisms.
\end{proof}

\begin{lemma}\label{V-gen}
Assume that $\lambda$-small (conical) limits commute in $\cv$ with $\lambda$-directed colimits.
Then, in an $\cm$-locally $\lambda$-generated $\cv$-category, $\lambda$-generated objects w.r.t. $\cm$ are closed under $\lambda$-small (conical) colimits.
\end{lemma}
\begin{proof}
The proof for $\lambda$-small conical colimits is the same as that of \ref{gen} and the proof for $\lambda$-small weighted limits is analogous (cf. \cite{BQR} 3.2).    
\end{proof} 

\begin{rem}\label{V-gen1}
{
\em
For the first claim, we do not need to assume that $\ck$ is $\cm$-locally $\lambda$-generated.
}
\end{rem} 

\begin{lemma}\label{V-quot}
Assume that finite conical limits commute with $\lambda$-directed colimits in $\cv$. Then, in an $\cm$-locally $\lambda$-generated $\cv$-category, $\lambda$-generated objects w.r.t. $\cm$ are closed under $\ce$-quotients.
\end{lemma}
\begin{proof}
Let $e:A\to B$ is in $\ce$ and $A\in\ca$. Express $B$ as a $\lambda$-directed colimit $m_i:B_i\to B$, $i\in I$, of $\ca$-objects and $\cm$-morphisms. Form pullbacks
$$
		\xymatrix@=3pc{
			P_i \ar [r]^{\bar{e}}\ar[d]_{\bar{m}_i} & B_i \ar[d]^{m_i }\\
			 A \ar [r]_{e}& B
		}
		$$ 
Since pullbacks commute with $\lambda$-directed colimits in $\cv$, they do it in $\cv^{\ca^{\op}}$ as well. Since $E$ preserves pullbacks and $\lambda$-directed colimits of $\cm$-morphisms, pullbacks commute with $\lambda$-directed colimits of $\cm$-morphisms in $\ck$. Hence $\bar{m}_i:P_i\to A$ is a $\lambda$-directed colimit. Since $\cm$ is $\lambda$-convenient, $m_i\in\cm$ for every $i\in I$. Hence $\bar{m}_i\in\cm$ for every $i\in I$ and, thus, $\bar{m}_{ij}:P_i\to P_j$ are in $\cm$ for every $i<j\in I$. Since $A$ is $\lambda$-generated w.r.t. $\cm$, $m_{i_0}$ splits for some $i_0\in I$. Thus there exists $s:A\to P_{i_0}$ such that $\bar{m}_{i_0}s=\id_A$. We have 
$$
m_{i_0}\bar{e}s=e\bar{m}_{i_0}s=e.
$$
The commutative square 		
$$
		\xymatrix@=3pc{
			A \ar [r]^{e}\ar[d]_{\bar{e}s} & B \ar[d]^{\id_B }\\
			 B_{i_0} \ar [r]_{m_{i_0}}& B
		}
		$$ 
has the diagonal $t:B\to B_{i_0}$ making $B$ a retract of $B_{i_0}$. Since $t$ is a coequalizer of $\id_{B_{i_0}}$ and $tm_{i_0}$, \ref{V-gen} implies that $B$ is $\lambda$-generated w.r.t. $\cm$.
\end{proof}

\begin{lemma}\label{V-ort}
If $\ck$ is $\cm$-locally $\lambda$-generated $\cv$-category then $\cm=(\ce_\lambda)^{\perp_\cv}$.
\end{lemma}
\begin{proof}	
It follows from \ref{V-fact}, \ref{ort} and \cite{LW} 5.4.
\end{proof} 

\begin{rem}\label{conical}
{
\em
(1) Given a small $\cv$-category, every $\cv$-functor $H:\ca\to\cv$ is a weighted colimit of representable $\cv$-functors. If $\cv$ is locally $\lambda$-presentable as a closed category (see \cite{K1} 5.5) then, for a small category $\ca$ with $\lambda$-small limits, any functor $H:\ca\to\cv$ preserving $\lambda$-small limits is a $\lambda$-filtered (and thus $\lambda$-directed) conical colimit of representable functors (see \cite{BQR} 4.5). In such a $\cv$, $\lambda$-small limits commute with $\lambda$-directed colimits (see \cite{BQR} 2.4).

(2) For a general $\cv$, let $H:\ca\to\cv$ preserve $\lambda$-small conical limits and consider the category $Y(\ca^{\op})\downarrow H$ of representable functors over $H$. This category is $\lambda$-filtered and let $H_\ast$ be its colimit in $\cv^\ca$ and $\gamma:H_\ast\to H$ be the comparison morphism. If $\ca$ is a $\lambda$-nest then, due to $(\ce,\cm)$ factorizations in $\ca$, the category $Y(\ca^{\op})\downarrow H$  has a cofinal subcategory $\cd$ whose morphisms are $Y(m):Y(A)\to Y(B)$ where $m$ are in $\cm$. $\cd$ is $\lambda$-directed and $H_\ast$ is it colimit of the projection $D:\cd\to \cv^\ca$ with a cocone $\varphi_d:Dd\to H_\ast$. 

Assume that $\cv$ is equipped with a $\lambda$-convenient factorization system $(\ce_\cv,\cm_\cv)$ and that hom-functors $\ca(A,-):\ca\to\cv$ send $\cm$-morphisms to $\cm_\cv$-morphisms. Them, for every $A$ in $\ca$, we have a $\lambda$-directed colimit $(\varphi_d)_A:Dd(A)\to H_\ast(A)$ of $\cm_\cv$-morphisms in $\cv$. Since the factorization system $(\ce,\cm)$ is $\lambda$-convenient, $(\varphi_d)_A$ are in $\cm$. 

Assume that $\cv$ is $\cm_\cv$-locally $\lambda$-generated and that $\ca$ has cotensors with $\lambda$-generated objects $V$ w.r.t. $\cm_\cv$. If $H$ preserves these cotensors then the argument from \cite{BQR} 4.5 yields that $\gamma$ is an isomorphism.  
}
\end{rem}

\begin{theo}\label{V-char}
Let $\cv$ be locally $\lambda$-presentable as a closed category. Then, for every $\cv$-category $\ck$ equivalent are:
\begin{enumerate}
\item $\ck$ is locally presentable,
\item $\ck$ is $\cm$-locally generated for some convenient $\cv$-factorization system $(\ce,\cm)$, and
\item $\ck$ is $\cm$-locally generated for every convenient $\cv$-factorization system $(\ce,\cm)$.
\end{enumerate}
\end{theo}
\begin{proof}
We proceed like in \ref{char}. In the implication $(2)\Rightarrow (1)$, we take $\Gen_\lambda(\ck)$ for $\ca$. Following \ref{V-gen}, $\cp$ is closed under finite tensors. Analogously as in \ref{V-ort}, we get that $\cp^\perp=\cp^{\perp_\cv}$. Hence $\Ort(\cp)$ consists of objects $\cv$-orthogonal to $\cp$. Finally, following \ref{conical}(1), every $L$ in $\Ort(\cp)$ is a $\lambda$-directed colimit of objects from $\Gen_\lambda(\ck)$.
\end{proof}

\begin{coro}\label{V-sketch}
Let $\cv$ be locally $\lambda$-presentable as a closed category and $\ck$ be an $\cm$-locally $\lambda$-generated $\cv$-category. Then $\ck$ is equivalent to the full subcategory of $\cv^{\Gen_\lambda(\ck)^{\op}}$ consisting of $\cv$-functors preserving $\lambda$-small limits and sending multiple pushouts of $\ce$-morphisms to multiple pullbacks.
\end{coro}
\begin{proof}
We proceed analogously as in \ref{sketch}. Note that \cite{AR} 1.33(8) is true in the enriched setting too but $m_i$ there should be
$$m_i:\colim\hom(Dd,-)\to\hom(\lim Dd,-).
$$ 
\end{proof}
\begin{nota}
{
\em
Again, we say that a $\cv$-category is \textit{locally $\lambda$-generated} if it is $\cm$-locally $\lambda$-generated for some $\cm$.  
}
\end{nota}

\begin{theo}\label{V-char1}
Assume that $\lambda$-small conical limits commute in $\cv$ with $\lambda$-directed colimits. Then a cocomplete $\cv$-category $\ck$ equipped with a $\lambda$-convenient $\cv$-factorization system $(\ce,\cm)$ is $\cm$-locally $\lambda$-generated iff its underlying category $\ck_0$ has a strong generator formed by $\lambda$-generated objects w.r.t. $\cm$.
\end{theo}
\begin{proof}
Using \ref{V-gen} and \ref{V-quot}, the closure $\ca$ of our strong generator under $\lambda$-small conical colimits and $\ce$-quotients consists of $\lambda$-generated objects w.r.t. $\cm$.
Then we follow the proof of \ref{char1}.
\end{proof}

\section{Enriched Gabriel-Ulmer duality}
\begin{defi}
{
\em
A \textit{$(\lambda,\cv)$-nest} is a small $\cv$-category $\ca$ equipped with a $\cv$-factorization system $(\ce_\ca,\cm_\ca)$ and having $\lambda$-small limits and multiple pullbacks of $\cm$-morphisms.
}
\end{defi}

\begin{exam}\label{nest}
{
\em
$\Gen_\lambda(\ck)^{\op}$ is a $(\lambda,\cv)$-nest for every $\cm$-locally $\lambda$-generated $\cv$-category $\ck$. This follows from \ref{V-gen}, \ref{V-quot} and the fact that $(\cm_\ca\cap\ca^\to,\ce_\ca\cap\ca^\to)$ is a $\cv$-factorization system on $\Gen_\lambda(\ck)^{\op}$.
}
\end{exam}

\begin{rem}
{
\em
Let $\cv$ be locally $\lambda$-presentable as a closed category. Then a small $\cv$-category $\ca$ is a $(\lambda,\cv)$-nest if and only is it is a $\lambda$-nest and $\cm_\ca$ is closed under $\lambda$-small cotensors.

We get this analogously to \ref{V-fact} applied to $\ca^{\op}$.
}
\end{rem}

\begin{nota}
{
\em
For a $(\lambda,\cv)$-nest $\ca$, $\Mod_\lambda(\ca)$ denotes the category of all $\cv$-functors $\ca\to\cv$ preserving $\lambda$-small limits and multiple pullbacks of $\cm$-morphisms.
}
\end{nota}

\begin{lemma}\label{fact-V}
Let $\cv$ be locally $\lambda$-presentable as a closed category. Then $\Mod_\lambda(\ca)$ is a locally presentable $\cv$-category equipped with a $\cv$-fac\-to\-ri\-za\-tion system $(\ce,\cm)$.
\end{lemma}
\begin{proof}
We proceed like in \ref{fact} with \cite{AR} 1.51 replaced by \cite{BQR} 7.3. We also use \ref{V-fact}.
\end{proof}

\begin{theo}\label{V-models}
Let $\cv$ be locally $\lambda$-presentable as a closed category. Then $\Mod_\lambda(\ca)$ is a $\cm$-locally $\lambda$-generated $\cv$-category for every $(\lambda,\cv)$-nest $\ca$.
\end{theo}
\begin{proof}
It follows from \ref{fact-V}, \ref{model} and \ref{V-fact}.
\end{proof}

\begin{defi}
{
\em
A \textit{morphism of $\lambda$-generated $\cv$-categories} $R: \ck \to \cl$ is a right $\cv$-adjoint preserving $\cm$-morphisms and $\lambda$-directed colimits of them. 

A \textit{morphism of $(\lambda,\cv)$-nests} $F: \ca \to \cb$ is a $\cv$-functor preserving $\lambda$-small limits, $\cm$-morphisms and multiple pullbacks of them.
}
\end{defi}

\begin{theo}\label{V-dual}
Let $\cv$ be locally $\lambda$-presentable as a closed category. Then \[ \Gen_\lambda:\LG _\lambda \leftrightarrows \N^{\op}_\lambda : \Mod_\lambda \] is a dual biequivalence between locally $\lambda$-generated $\cv$-categories and $(\lambda,\cv)$-nests.
\end{theo}
\begin{proof}
It follows from \ref{dual}, \ref{nest} and \ref{V-models}.
\end{proof}

\begin{rem}
{
\em
(1) Our duality restricts to the enriched Gabriel-Ulmer duality between locally $\lambda$-presentable $\cv$-categories and small $\cv$-categories with $\lambda$-small limits given in \cite{K1}. 

(2) \ref{duality} (2) applies to the enriched case as well.
}
\end{rem}

\section{Banach spaces}
 
Let $\CMet$ be the category of (generalized) metric spaces and nonexpanding maps. This category
is symmetric monoidal closed and locally $\aleph_1$-presentable as a closed category (see \cite{AR2} 2.3(2) and 4.5(2)). But it is not locally $\aleph_0$-presentable. In fact, only the empty space is $\aleph_0$-presentable in $\CMet$ (see \cite{AR2} 2.7 (1)). $\CMet$ has a strong generator consisting of a one-point space $1$ and of two-point spaces $2_{\eps}$ where the two points have the distance $\eps>0$.

The category $\Ban$ of Banach spaces and linear maps of norm $\leq 1$ is enriched over $\CMet$. Moreover, it is locally $\aleph_1$-presentable $\CMet$-category (see \cite{AR2} 6.3).

\begin{rem}
{
\em
(1) Epimorphisms in $\CMet$ or $\Ban$ coincide with dense maps. See \cite{P} 1.15 for $\Ban$ and the argument for $\CMet$ is analogous. Both in $\CMet$ and $\Ban$, there is a factorization system $(\ce,\cm)$ where $\ce$ consists of dense maps and $\cm$ of isometries (see \cite{AR2} 3.16(2)). Hence isometries coincide, both in $\CMet$ and in $\Ban$ with strong monomorphisms. Both $\CMet$ and $\Ban$ are $\ce$-cowellpowered and, from the description of directed colimits (see \cite{AR2} 2.5), it follows that $(\ce,\cm)$ is, in the both cases, $\aleph_0$-convenient.

(2) Both in $\CMet$ and $\Ban$, $\aleph_0$-generated objects w.r.t. $\cm$ coincide with approximately $\aleph_0$-generated objects in the sense of \cite{AR2} (see 5.11(3) in this paper). In $\CMet$, approximately $\aleph_0$-generated finite metric spaces are precisely discrete ones (see \cite{AR2} 5.18 and 5.19). 

(3) $\Ban$ is $\cm$-locally $\aleph_0$-generated (see \cite{AR2} 7.8). Every finite-dimensional Banach space is approximately $\aleph_0$-generated (\cite{AR2} 7.6). 

(4) The concept of a finite weight does not make sense in $\CMet$ because only $\emptyset$ is $\aleph_0$-presentable. Thus, under finite limits in $\CMet$, we understand finite conical limits, cotensors with finite metric spaces, and their combinations. Here, cotensors with finite metric spaces can be replaced by $\varepsilon$-pullbacks (see \cite{AR2} 4.6). But 
in $\CMet$, $\aleph_0$-generated objects w.r.t. $\cm$ are not closed under these finite colimits (see \cite{AR2} 5.20). Hence, following \ref{V-gen} and \cite{AR2} 5.20 and 4.1(4), finite limits do not commute with directed colimits in $\CMet$.
}
\end{rem}

\begin{rem}
{
\em
A metric space is called \textit{convex} if for every points $x$ and $y$ there is a point $z$ such that $d(x,z)+d(z,y)=d(x,y)$. A subset $S$ of a metric space $A$ is called a \textit{metric segment} if for every two distinct points $a\neq b$ there is an isometry $f:[0,d(a,b)]\to A$ from a closed interval on $\Bbb R$ such that $f(0)=a$, $f(d(a,b))=b$ and $f([0,d(a,b)])=S$. A complete metric space is convex iff every distinct points are connected by a metric segment (see \cite{Br}).
The proof consists in creating a dense set of points between $a$ and $b$ and taking its completion. Hence it also applies to $d(a,b)=\infty$ where the interval is $[0,\infty]$ in $\Bbb R$ with $\infty$ added.

}
\end{rem}

\begin{lemma}\label{convex}
For every $\delta>0$, $\CMet(2_{\delta},-):\CMet\to\CMet$ preserves directed colimits of convex spaces and isometries.
\end{lemma}
\begin{proof}
Let $k_i:K_i\to K$, $i\in I$ be a directed colimit of convex complete metric spaces and isometries. Consider $f:2_\delta\to K$ and choose $\eps>0$. Denote the two points of $2_\delta$ as $x$ and $y$. There is $i\in I$ and $a,b\in K_i$ such that $d(a,fx), d(b,fy)\leq\frac{\eps}{2}$. Then
$$
d(a,b)\leq d(a,fx) + d(fx,fy) + d(fy,b)\leq \eps + \delta.
$$
Let $S$ be the metric segment in $K_i$ connecting $a$ and $b$. Choose $a',b'\in S$ such that $d(a,a'),d(b,b')\leq\frac{\eps}{2}$. Then $d(a',b')\leq\delta$. Let $g:2_\delta\to K_i$ be given by $fx=a'$ and $fy=b'$. Then $d(f,k_ig)$ because $d(a',fx)\leq d(a',a) + d(a,fx)\leq\eps$ and similarly $d(b',fy)\leq\eps$. 
\end{proof}

\begin{exam}
{
\em
There is an approximately $\aleph_0$-generated Banach spaces which is not finite-dimensional.
Consider the complete metric space $A=\{0\}\cup\{\frac{1}{n};n=1.2.,\dots\}$. Let $A_m=\{\frac{1}{n};n=1,2,\dots,m\}$ and $r_m:A_m\to A$ be the inclusion. Let $u_m:A\to A_m$ be the identity on $A_m$ and send $A\setminus A_m$ to $\frac{1}{m}$. We have $r_mu_m\sim_{\frac{1}{m}}\id_{A_m}$. 

Let $F:\CMet\to\Ban$ be the left adjoint to the unit ball functor $U:\Ban\to\CMet$ (see \cite{AR2} 4.5(3)). This adjunction is enriched and thus $F(r_mu_m)\sim_{\frac{1}{m}}\id_{F(A_m)}$. 
Since $F(1)=\Bbb C$, $F$ sends finite discrete metric spaces to finite-dimensional Banach spaces.
We have surjective maps $D_m\to A_m$ where $D_m$ is discrete space of $m$ points. Since $U$ preserves isometries, $F$ preserves $\ce$-maps. Thus $F(A_m)$ are finite-dimensional Banach spaces. Following \cite{AR2}7.7(1), $F(A)$ is approximately $\aleph_0$-generated.

Since $u_m$ are surjective, $F(u_m)$ are dense. Thus $F(A)$ cannot be finite-dimensional because it has dense maps to $m$-dimensional Banach spaces for every $m$.
}
\end{exam}

\begin{lemma}\label{e-push}
An $\eps$-pushout in $\Ban$ is
$$
\xymatrix@=3pc{
A\ar[r]^{f} \ar[d]_{g} & B\ar[d]^{\overline{g}} \\
C \ar[r]_{\overline{f}} & B\oplus_{f,g,\eps} C
}
$$
where $B\oplus_{f,g,\eps} C$ is the coproduct $B\oplus C$ endowed with the norm
$$
\pa (x,y)\pa=\inf\{\pa b\pa + \pa c\pa + \eps\pa a\pa\backslash\, x= b+f(a),y=c-g(a)\}.
$$
\end{lemma}
\begin{proof}
In the special case of $B\oplus_{f,\id_A,\eps} C$, it is \cite{G} 2.1. The general case is analogous.
\end{proof}

\begin{coro}
\em
The dual of the full subcategory $\Ban_{\fd}$ of $\Ban$ consisting of finite-dimensional Banach spaces 
is a $(\aleph_0,\CMet)$-nest. 
\end{coro}
\begin{proof}
Finite-dimensional Banach spaces are closed under dense quotients. Following \cite{AR2} 4.6 and \ref{e-push}, they are closed under finite colimits.
\end{proof}

\begin{theo}
$\Ban$ is equivalent to  $\Mod_{\aleph_0}(\Ban_{\text{fd}}^{\op})$.
\end{theo}
\begin{proof}
Following \ref{V-sketch}, $\Ban$ is equivalent to the full subcategory of $\Mod_{\aleph_0}\Ban_{\fd}^{\op}$. Consider $H$ in $\Mod_{\aleph_0}\Ban_{\fd}^{\op}$ and follow \ref{conical}(2). For every Banach space $X$, $H(X)$ is the directed colimit 
$\colim_d Dd(X)$ of complete metric spaces $Dd(X)$ and isometries. Every $Dd(X)$ is $\Ban(A,X)$ for some finite-dimensional Banach space $A$. Since these complete metric spaces are convex, we have
$$
\CMet(2_\delta,\colim_d Dd(X))\cong\colim _d \CMet(2_\delta,Dd(X)).
$$
Since $1$ and $2_\delta$, $\delta>0$ form a strong generator in $\CMet$, $\gamma:H_\ast\to H$ is an isomorphism (following the argument from \cite{BQR} 4.5).
\end{proof}

\begin{rem}
{
\em
The category $\Ban^\to$ does not seem to be $\cm$-locally $\aleph_0$-generated -- \ref{V-arrow1} does not apply to it. Hence \cite{AR2} cannot be immediately applicable to the construction of approximately $\aleph_0$-saturated objects in $\Ban^\to$ like in $\Ban$. The existence of such objects in $\Ban^\to$ was proved in \cite{GK}.
}
\end{rem}

\begin{rem}
{
\em
Let $\CCMet$ be the category of convex complete metric spaces and non-expansive maps. $\CCMet$ is 
an injectivity class in $\CMet$ given by $2_\delta\to [0,\delta]$, $\delta>0$, sending the two points of $2_\delta$ to the end-points of $[0,\delta]$. Thus $\CCMet$ is a full weakly reflective subcategory of $\CMet$. A weak reflection of a complete metric space $A$ to $\CCMet$ is constructed as follows. If $a,b\in A$ are not connected by a metric segment of length $d(a,b)$,
we glue $[0,d(a,b)]$ to $A$ by identifying $0$ with $a$ and $d(a,b)$ with $b$. The distances of points of distinct added segments are $\infty$. We repeat the procedure and add metric segments $[0,\infty]$.  We proceed by induction and the union $\cup_n A_n$, $n=1,2,\dots$ is a desired weak reflection. But $\CCMet$ is not reflective in $\CMet$ because it is not closed under equalizers. Indeed, let $C$ be the circle of radius $1$. Then the equalizer of the identity and the axial symmetry on $C$ is $2_2$. Even, this equalizer does not exist in $\CCMet$ at all. 
$\CCMet$ does not have coproducts - for instance $1 + 1$ does not exist in $\CCMet$.

$\CCMet$ is closed in $\CMet$ under $\aleph_1$-directed colimits. Indeed, let $k_i:K_i\to K$, $i\in I$  be an $\aleph_1$-directed colimit of convex complete metric spaces in $\CMet$. Let $a,b\in K$. Following \cite{AR2} 2.5(1), there is $i\in I$ and $a',b'\in K_i$ such that $d(a',b')=d(a.b)$, $k_ia'=a$ and $k_ib'=b$. There is $c\in K_i$ such that $d(a',c)+d(c,b')=d(a',b')$. Since $d(a,fc)+d(fc,b)\leq d(a',c)+d(c,b')=d(a,b)$, we have $d(a,fc)+d(fc,b)=d(a,b)$. Thus $K$ is convex. Analogously, using \cite{AR2} 2.5(2), we prove that $\CCMet$ is closed under directed colimits of isometries in $\CMet$. We do not know whether the segments $[0,\delta]$ are $\aleph_0$-generated w.r.t. isometries in $\CCMet$.  

The tensor product $A\otimes B$ of convex complete metric spaces $A$ and $B$ is convex and complete. Recall that $A\otimes B$ is $A\times B$ with the $+$-metric
$$
d(a,b),(a',b'))=d(a,a')+d(b,b').
$$
Indeed, let $a''\in A$ and $b''\in B$ satisfy $d(a,a'')+d(a'',a')=d(a,a')$ and 
$d(b,b'')+d(b'',b')=d(b,b')$ . Then
\begin{align*}
d((a,b),(a'',b''))+d((a'',b''),(a',b')) &=d(a,a'')+d(b,b'')+d(a'',a')+d(b'',b')\\
&=d(a,a')+d(b,b')=d((a,b),(a',b')).
\end{align*}
We do not know whether $\CCMet$ is symmetric monoidal closed.
}
\end{rem}

\begin{rem}\label{calg}
{
\em
Let  $\CAlg$ be the category of $C^\ast$-algebras and $\CCAlg$ the category of commutative $C^\ast$-algebras. The forgetful functor $U:\CAlg\to\Ban$ preserves limits, isometries and $\aleph_1$-directed colimits. Thus it has a left adjoint $F$. The same holds for $U:\CCAlg\to\Ban$. The unit $\eta_B:B\to UFB$ is a linear isometry. Thus $F$ is faithful. In the commutative case, 
this left adjoint was described in \cite{S} and called the Banach-Mazur functor.
 
The forgetful functor $U:\CAlg\to\Ban$ even preserves directed colimits. In fact,
directed colimits in $\Ban$ are calculated like in $\CMet$, i.e., as a completion of a directed colimit in $\Met$. However,
the same holds in $\CAlg$ because one completes the directed colimit of $\ast$-algebras. Indeed, if $x=\lim_n x_n$ and $y=\lim_m y_m$ are in this completion then both $x\cdot y=\lim_{n,m} (x_n\cdot y_m)$ and  $x^\ast = \lim_n x_n^\ast$ are there.

Free (commutative) $C^\ast$-algebras over finite-dimensional Banach spaces are $\aleph_0$-ap-generated and, since $U$ is conservative, they form a strong generator in $\CAlg$.
The category of (commutative) $C^\ast$-algebras is locally $\aleph_1$-presentable and monadic over $\Set$ (see \cite{PR}). Since epimorphisms are surjective in $\CCAlg$ (see \cite{HN}), (epi, strong mono)-factorization system on $\CAlg$ is (surjective, dense) one and it is $\aleph_0$-convenient. Hence $\CAlg$ is a cocomplete $\CMet$-category with a $\aleph_0$-convenient $\cv$-factorization system (surjective, isometry) and with a strong generator consisting of $\aleph_0$-generated objects w.r.t. isometries. But \ref{V-char1} does not apply to $\CAlg$ and we expect that this category is not isometry-locally $\aleph_0$-generated.
}
\end{rem}

\end{document}